\documentclass[11pt,a4paper]{article}
\usepackage{amsmath,amsthm,amsfonts,amssymb}

\theoremstyle{plain}

\newtheorem{teo}{Theorem}[section]
\newtheorem{lem}[teo]{Lemma}
\newtheorem{pro}[teo]{Proposition}
\newtheorem{cor}[teo]{Corollary}
\newtheorem{eje}[teo]{Example}

\theoremstyle{definition}
\newtheorem*{dfn}{Definition}
\newtheorem{que}[teo]{Question}

\numberwithin{equation}{section}

\newcommand{\F}{\mathbb{F}}
\newcommand{\N}{\mathbb{N}}
\newcommand{\Z}{\mathbb{Z}}

\DeclareMathOperator{\Zen}{Z}
\DeclareMathOperator{\cen}{C}
\DeclareMathOperator{\br}{br}
\DeclareMathOperator{\Span}{span}

\renewcommand{\phi}{\varphi}

\title{On $w$-maximal groups}

\author{Jon Gonz\'alez-S\'anchez and Benjamin Klopsch}

\begin{document}

\maketitle

\begin{abstract}
Let $w = w(x_1,\ldots , x_n)$ be a word,
i.e.\ an element of the free group $F = \langle x_1, \ldots,
x_n\rangle$ on $n$ generators $x_1, \ldots, x_n$.  The verbal subgroup
$w(G)$ of a group $G$ is the subgroup generated by the set $\{ w
(g_1,\ldots ,g_n)^{\pm 1} \mid g_i \in G, \ 1\leq i\leq n \}$ of all $w$-values
in $G$.  We say that a (finite) group $G$ is \emph{$w$-maximal} if $\lvert G:w(G)
  \rvert > \lvert H:w(H) \rvert$ for all proper subgroups $H$ of $G$ and that $G$ is 
  \emph{hereditarily $w$-maximal} if every subgroup of $G$ is $w$-maximal. 
  In this text we study $w$-maximal and hereditarily $w$-maximal (finite) groups.
  \end{abstract}

\section{Introduction}

Let $p$ be a prime.  In~\cite{Th69}, Thompson observed that, if $G$ is a
finite $p$-group such that $\lvert G:[G,G] \rvert > \lvert H:[H,H]
\rvert$ for all proper subgroups $H$ of $G$, then the nilpotency class
of $G$ is at most $2$.  This insight prompted Laffey to prove that,
for $p>2$, the minimum number of generators $d(G)$ of a finite
$p$-group $G$ is bounded by $r$, where $p^r$ is the maximal order of a
subgroup of exponent $p$ in $G$; see~\cite{La73}.

Properties of what are known as $d$-maximal $p$-groups form a key
ingredient of Laffey's argument; a group $G$ is $d$-maximal if $d(H) <
d(G)$ for all proper subgroups $H$ of $G$.  The minimum number of
generators for a finite $p$-group $G$ is given by $d(G) = \log_p
\lvert G : \Phi(G) \rvert$, where $\Phi(G) = G^p [G,G]$ denotes the
Frattini subgroup of $G$.  Hence a finite $p$-group $G$ is $d$-maximal
if $\lvert G : \Phi(G) \rvert > \lvert H:\Phi(H) \rvert$ for all
proper subgroups $H$ of $G$.

In the context of regular representations of finite groups, Kahn
proved, for $p>2$, that every $d$-maximal finite $p$-group $G$ has
nilpotency class at most~$2$; see~\cite{Ka91}.  In fact, he showed
that in such a group $G$ the derived subgroup $[G,G]$ is of exponent
$p$ and contained in the center $\Zen(G)$ of $G$.  Subsequently
properties of $d$-maximal finite $p$-groups were investigated by Kahn
as well as Minh, e.g.\ see \cite{Ph96}.  More recently, a similar
class of groups was studied by the first author in order to bound the
index of the agemo subgroup of a finite $p$-group in terms of the
number of elements of order $p$.  In particular, he proved for $p>2$
that, if $G$ is a finite $p$-group such that $\lvert G : G^p
\gamma_{p-1}(G) \rvert > \lvert H : H^p \gamma_{p-1}(H) \rvert$ for
every proper subgroup $H$ of $G$, then the nilpotency class of $G$ is
bounded by $p-1$; see~\cite{Go09}.

The aim of this paper is to explore this circle of fruitful ideas in a
more general framework.  For this we introduce and explore the new
concept of a $w$-maximal group, which is briefly referred to
in~\cite{Go09}.  Let $w(\mathbf{x}) = w(x_1,\ldots , x_n)$ be a word,
i.e.\ an element of the free group $F = \langle x_1, \ldots,
x_n\rangle$ on $n$ generators $x_1, \ldots, x_n$.  The verbal subgroup
$w(G)$ of a group $G$ is the subgroup generated by the set $\{ w
(\mathbf{g})^{\pm 1} \mid \mathbf{g} \in G^{(n)} \}$ of all $w$-values
in $G$.

\begin{dfn}
  We say that a group $G$ is \emph{$w$-maximal} if $\lvert G:w(G)
  \rvert > \lvert H:w(H) \rvert$ for all proper subgroups $H$ of $G$.
\end{dfn}

The classes of groups referred to above are instances of $w$-maximal
groups for the special words $w = [x,y] = x^{-1}y^{-1}xy$, $w = x^p
[y,z]$ and $w = x^p [y_1,\ldots ,y_{p-1}]$, respectively.  Our study
of $w$-maximal groups for more general values of $w$ sheds more light
on existing theorems and leads to a number of new results.  Indeed,
Thompson's original theorem generalises to $w$-maximal groups for many
words $w$.  We also include the outcomes of our study of hereditarily
$w$-maximal groups, i.e.\ groups with the property that all their
subgroups are $w$-maximal.  Finally we suggest a range of questions
concerning the structure of $w$-maximal and hereditarily $w$-maximal
groups, in order to stimulate further research in this direction.  We
use a mixture of methods, involving, for instance, classical results
of Iwasawa~\cite{Iw41,Iw41b} and techniques from the theory of finite
$p$-groups and their inverse limits.

\medskip

The organisation of the paper is the following.  In
Section~\ref{sec:basic} we introduce the concept of $w$-breadth
and a partial ordering on the class of $w$-maximal groups.  We
show that, if $w$ is a commutator word, then every ascending chain
of $w$-maximal finite groups with respect to the defined ordering
becomes stationary; see Corollary~\ref{cor:stationary}.  A key
ingredient is Theorem~\ref{teo:breadth} which characterises words
$w$ admitting only finitely many (isomorphism types of) finite
groups of any fixed $w$-breadth.  Our results lead to a natural
classification problem. Section~\ref{sec:w_max_p_groups} focuses
on $w$-maximal finite $p$-groups.  We define and illustrate the
concept of interchangeability.  Theorem~\ref{thm:thompson} shows
that, if $w$ is interchangeable in a $w$-maximal finite $p$-group
$G$, then the verbal subgroup $w(G)$ is central in $G$.  The proof
is based on Thompson's original argument in \cite{Th69}.  We state
various questions, of which we highlight here the problem of
classifying those $w$ which are interchangeable in every
$w$-maximal finite $p$-group.  We also give some applications of
Theorem~\ref{thm:thompson}.  These provide, in any finite
$p$-group, a lower bound for the maximum size of subgroups with
certain properties in terms of the size of a quotient with similar
properties; e.g.\ see Propositions~\ref{pro:large_subgroups} and
Proposition~\ref{pro:epsilon}.  Corollary~\ref{cor-2} is an
extension of the result of Laffey mentioned above.  In
Section~\ref{sec:Lie}, we investigate $d$-maximal finite
$p$-groups by Lie theoretic means.  For $p$ odd, the
abelianisation $G/[G,G]$ and the commutator subgroup $[G,G]$ of a
$d$-maximal finite $p$-group $G$ are elementary abelian.  We show
that, furthermore, $\lvert [G,G] \rvert \leq p^{d(G)-2}$ and
construct examples which suggest that this bound may be best
possible.  In Section~\ref{sec:hereditarily_w} we study
hereditarily $w$-maximal groups.  Theorem~\ref{thm:example}
provides examples of non-trivial hereditarily $w$-maximal finite
groups, i.e.\ hereditarily $w$-maximal finite groups $G$ such that
$w(G) \not = 1$, for the higher derived words $w = \delta_k$, $k
\geq 2$.  Finally, in Section~\ref{sec:hereditarily_d} we
investigate hereditarily $d$-maximal groups.  In order to avoid
unwieldy examples, it is natural to restrict attention to the
class of residually-finite groups, which, in fact, reduces further
to the class of finite groups. Theorem~\ref{thm:classif} provides
a complete classification of hereditarily $d$-maximal finite
groups.

\medskip

\noindent \textit{Notation:} Short explanations of possibly
non-standard notation and terminology are given in the text.  Let
$G$ be a group. For $n \in \N$ we write $G^{(n)}$ for the $n$th
cartesian power of $G$, whereas $G^n$ refers to the subgroup
generated by all $n$th powers of elements of $G$.  If $G$ is a
topological group, then invariants such as $d(G)$ or $\br_w(G)$
are tacitly defined in terms of closed subgroups; e.g.\ $d(G)$ is
the minimum number of topological generators.

%%%

\section{Basic properties of $w$-maximal groups} \label{sec:basic}

Throughout this section, let $n \in \N$ and let $w = w(\mathbf{x})$ be
a fixed word, i.e.\ an element of the free group $F = \langle
x_1,\ldots,x_n \rangle$.  Our first lemma collects two elementary, but
useful properties of $w$-maximal groups.

\begin{lem}
  \label{lem:properties}
  Let $w$ be a word and let $G$ be a finite group.
  \begin{enumerate}
  \item[\textup{(a)}] If $w(G)=1$, then $G$ is $w$-maximal.
  \item[\textup{(b)}] If $G$ is $w$-maximal and $N \trianglelefteq G$
    with $N \subseteq w(G)$, then $G/N$ is $w$-maximal.
  \end{enumerate}
\end{lem}

\begin{proof}
  These properties are clear.
\end{proof}

Based on Lemma~\ref{lem:properties}~(b), we introduce a partial
ordering on the class of $w$-maximal groups.  Let $G$ and $H$ be
$w$-maximal groups.  We say that $H$ precedes $G$, in symbols $H
\preceq_w G$, if there exists an epimorphism $f: G \to H$ such
that $\ker f \leq w(G)$.  The binary relation $\preceq_w$ defines
a partial ordering on the class of $w$-maximal groups.

For a group $G$, we define the \textit{$w$-breadth} of $G$ as
$$
\br_w(G) := \sup \{ \lvert H:w (H) \rvert \mid H \leq G \}.
$$
The next theorem shows that, if $w$ is a commutator word, then groups
of bounded $w$-breadth can be controlled.

\begin{teo}
  \label{teo:breadth}
  Let $w = w(\mathbf{x})$ be an element of the free group $F = \langle
  x_1, \ldots, x_n \rangle$.  Then the following assertions are
  equivalent:
  \begin{enumerate}
  \item[\textup{(1)}] $w$ is a commutator word, i.e.\ $w \in [F,F]$;
  \item[\textup{(2)}] for every natural number $m$ there are only
    finitely many (isomorphism classes of) finite groups $G$ with
    $\br_w(G) \leq m$.
  \end{enumerate}
\end{teo}

\begin{proof}
  First suppose that $w$ is a commutator word, and let $m$ be a
  natural number.  If $G$ is a group with $\br_w(G) \leq m$, then the
  prime divisors of $\lvert G \rvert$ are less than or equal to~$m$.
  Indeed, if $S$ is a non-trivial Sylow $p$-subgroup of $G$, then $p
  \leq \lvert S:w(S) \rvert \leq m$.

  Consider now a prime number $p$ less than or equal to $m$.  We claim
  that there exist only finitely many (isomorphism classes of)
  $p$-groups $P$ with $\br_w(P) \leq m$.  This will show that there is
  a uniform bound on the sizes of the Sylow $p$-subgroups of a finite
  group $G$ with $\br_w(G) \leq m$.  Consequently, there is a uniform
  bound on the order of a finite group $G$ with $\br_w(G) \leq m$, and
  there are only finitely many finite groups with this property.

  For a contradiction, assume that there exist an infinite number of
  pairwise non-isomorphic finite $p$-groups with $w$-breadth bounded
  by $m$.  These groups and the possible epimorphisms between them
  give rise to an inverse system.  Since the inverse limit of a
  non-empty system of finite sets is non-empty, one obtains an
  infinite pro-$p$ group $G$ with $\br_w(G) \leq m$.  Since $w \in
  [F,F]$, we have $d(H) \leq \log_p \lvert H : w(H) \rvert \leq \log_p
  m$ for every closed subgroup $H$ of $G$.  Consequently the group $G$
  has finite rank, equivalently $G$ is $p$-adic analytic;
  cf.~\cite[Chapters 8 and~9]{DidSMaSe99}.  Therefore there exists a
  uniformly powerful open subgroup $U$ of $G$; see
  \cite[Corollary~4.3]{DidSMaSe99}.  But in this case, $\lvert U^{p^k}
  : w(U^{p^k}) \rvert \geq \lvert U^{p^k}:[U^{p^k},U^{p^k}] \rvert$
  tends to infinity as $k$ tends to infinity.  This is in
  contradiction to $\br_w(G) \leq m$.

  Conversely, suppose that $w$ is not a commutator word, i.e.\ $w \not
  \in [F,F]$.  Let $p$ be a prime and consider a free $\Z_p$-module
  $M$ with basis $\{x_1,\ldots ,x_d\}$, where $\Z_p$ denotes the ring
  of $p$-adic integers.  Then $w(M)$ is a non-trivial characteristic
  subgroup of $M$ and hence $w(M) = p^r M$ for some $r \in \N_0$.  In
  fact, for all submodules $U$ of $M$ we have $w(U) = p^r U$ and
  consequently $\lvert U:w(U) \rvert \leq \lvert M:w(M) \rvert$.  Set
  $m := p^{rd} = \lvert M:w(M) \rvert$.  Then for every open submodule
  $K$ of $M$ with $K \subseteq w(M)$, the finite quotient $M/K$ has
  $w$-breadth bounded by $m$.  This shows that there are infinitely
  many (isomorphism classes of) finite abelian groups $G$ with
  $\br_w(G) \leq m$.
\end{proof}

Motivated by the proof of Theorem~\ref{teo:breadth}, we record

\begin{pro}
  Let $w = w(\mathbf{x})$ be an element of the free group $F = \langle
  x_1, \ldots, x_n \rangle$.  Let $m \in \N$ and $p$ a prime.  Let
  $\mathcal{G}_m(p)$ denote the class of all pro-$p$ groups $G$ with
  $\br_w(G) \leq m$.
  \begin{enumerate}
  \item[\textup{(1)}] If $w \in [F,F]$, then $\mathcal{G}_m(p)$ consists
    of finitely many isomorphism classes of finite $p$-groups.
  \item[\textup{(2)}] If $w(F) F^p [F,F] = F$, then $\mathcal{G}_m(p)$
    consists of all pro-$p$ groups.
  \item[\textup{(3)}] Suppose that $w(F) F^{p^{k+1}} [F,F] = F^{p^k}
    [F,F]$ with $k \in \N$.  Then $\mathcal{G}_m(p)$ consists of
    isomorphism classes of $p$-adic analytic pro-$p$ groups.
    Conversely, if $G$ is a $p$-adic analytic pro-$p$ group of
    dimension $d$ and if $m \geq p^{kd}$, then there exists an open
    subgroup $U$ of $G$ such that all subgroups of $U$ are contained in
    $\mathcal{G}_m(p)$.
  \end{enumerate}
\end{pro}

\begin{proof}
  (1) This follows immediately from Theorem~\ref{teo:breadth}.

  (2) Suppose that $w(F) F^p [F,F] = F$.  Then $w(H) = H$ for any
  finite $p$-group $H$, because the $w$-values in $H$ generate $H$
  modulo the Frattini subgroup $H^p [H,H] = \Phi(H)$.  The claim
  follows.

  (3) Let $G$ be in $\mathcal{G}_m(p)$.  Since $w(F) \subseteq F^p
  [F,F]$, we have $d(H) \leq \log_p \lvert H : w(H) \rvert
  \leq \log_p m$ for every closed subgroup $H$ of $G$.  Thus $G$ has
  finite rank and is $p$-adic analytic.

  Conversely, let $G$ be a $p$-adic analytic pro-$p$ group of
  dimension $d$ and suppose that $m \geq p^{kd}$.  Take for $U$ a
  uniformly powerful open subgroup of $G$.  A simple collection
  process allows us to write $w = x_1^{e_1} \cdots x_r^{e_r} c$, where
  $e_1, \ldots, e_r \in \Z$ and $c \in [F,F]$.  Since $w \not \in
  [F,F]$, at least one of the exponents $e_i$ is non-zero, and we put
  $e := \gcd \{e_1, \ldots, e_r\}$.  Then $F^e \subseteq w(F)$ and
  $w(F) [F,F] = F^e [F,F]$.  From $w(F) F^{p^{k+1}} [F,F] = F^{p^k}
  [F,F]$ it follows that $p^k$ is the highest $p$-power dividing $e$.
  Hence for any closed subgroup $H$ of $U$ one has $H^{p^k} = H^e
  \subseteq w(H)$.  Moreover, for any closed subgroup $H$ of $U$, one
  has $\mu \{ h^{p^k} \mid h \in H \} = p^{-k} \mu(H)$, where $\mu$
  denotes the Haar measure on $H$; see \cite[Lemma~3.4]{Kl05}.  Thus
  $\lvert H : w(H) \rvert \leq \lvert H : H^{p^k} \rvert \leq p^{kd}$
  for any open and, by passing to the appropriate limit, for any
  closed subgroup $H$ of $U$.
\end{proof}

From Lemma~\ref{lem:properties} and Theorem~\ref{teo:breadth} one
can easily deduce

\begin{cor} \label{cor:stationary}
  Let $w = w(\mathbf{x})$ be a commutator word in a free group $F =
  \langle x_1, \ldots, x_n \rangle$, i.e.\ let $w \in [F,F]$.  Then
  every ascending chain
  \begin{equation*}
    H_1 \preceq_w H_2 \preceq_w \ldots \preceq_w H_i \preceq_w \ldots
  \end{equation*}
  of $w$-maximal finite groups is eventually stationary.

  In particular, if $H$ is a $w$-maximal finite group, then there
  exists a $w$-maximal finite group $G$ such that $H \preceq_w G$ and
  $G$ is maximal with respect to the partial ordering $\preceq_w$.
\end{cor}

\begin{proof}
  We argue by contradiction.  If
  \begin{equation*}
    H_1 \preceq_w H_2 \preceq_w \ldots \preceq_w H_i \preceq_w \ldots
  \end{equation*}
  is a strictly increasing chain of $w$-maximal groups, then there
  exist infinitely many finite groups whose $w$-breadth is bounded by
  $\lvert H_1:w (H_1) \rvert$ in contradiction to
  Theorem~\ref{teo:breadth}.
\end{proof}

This leads to a natural classification problem.

\begin{que} \label{qu:max-w-max} Let $F$ be the free group on
  generators $x_1, \ldots, x_n$, and let $w$ be a commutator word in
  $F$, i.e.\ let $w \in [F,F]$.  Can one classify $w$-maximal groups
  of a given $w$-breadth $m$ which are maximal with respect to the
  partial ordering $\preceq_w$?
\end{que}

%%%

\section{Properties of $w$-maximal $p$-groups and
  applications} \label{sec:w_max_p_groups}

In this section we focus on $w$-maximal finite $p$-groups which are
somehow easier to control than $w$-maximal finite groups.  The
following concept will be useful throughout this section.  Let $w =
w(\mathbf{x})$ be a word and let $G$ be a finite $p$-group.  We say
that $w$ is \textit{interchangeable} in $G$ if for every normal
subgroup $N$ of $G$,
\begin{equation*}
  [w (N),G] \leq [N,w (G)] \cdot [w (G),G]^p[w (G),G,G].
\end{equation*}
Observe that, if $w$ is interchangeable in $G$, then $w$ is
interchangeable in every quotient of $G$.  The next lemma provides a
considerable supply of words which are interchangeable in every finite
$p$-group.

\begin{lem}
  \label{swap}
  Let $G$ be a finite $p$-group, and let $w$ be equal to one of the
  group words
  \begin{enumerate}
  \item[\textup{(i)}] $\gamma_k = [y_1,\ldots ,y_k]$ for some $k \in
    \N$,
  \item[\textup{(ii)}] $x^{p^i} [y_1,\ldots ,y_k]$ for some $i, k \in
    \N$ with $k \leq p-1$,
  \item[\textup{(iii)}] $x^{p^i} [y_1,\ldots ,y_p]$ for some $i \in \N$
    with $i\geq 2$.
  \end{enumerate}
  Then $w$ is interchangeable in $G$.
\end{lem}

\begin{proof}
  Let $N$ be a normal subgroup of $G$.

  (i) We start with the case $w = \gamma_k = [y_1,\ldots ,y_k]$, where
  $k \in \N$.  Then, by \cite[Lemma~4-9]{Kh73}, one has
  \begin{equation}
    [\gamma_k(N),G] \leq [G,_kN] \leq [\gamma_k(G),N] = [N,\gamma_k(G)].
  \end{equation}

  (ii) Next suppose that $w = x^{p^i}[y_1,\ldots ,y_k]$, where $i,k
  \in \N$ with $k \leq p-1$.  By (i), we have $[\gamma_k(N),G] \leq
  [N,\gamma_k(G)]$.  For $p^i$th powers \cite[Theorem~2.4]{FeGoJa08}
  yields
  $$
  [N^{p^i},G] \equiv [N,G^{p^i}] \pmod{\gamma_{p+1}(G)}.
  $$
  Therefore we conclude that
  $$
  [N^{p^i}\gamma_k(N),G] \leq [N, G^{p^i}\gamma_k(G)] \cdot
  [\gamma_k(G),G,G].
  $$

  (iii) Finally suppose that $w = x^{p^i}[y_1,\ldots ,y_p]$, where $i
  \in \N$ with $i \geq 2$.  Again $[\gamma_p(N),G] \leq
  [N,\gamma_p(G)]$.  Since $i\geq 2$, we conclude from
  \cite[Theorem~2.4]{FeGoJa08} that
  $$
  [N^{p^i},G] \equiv [G^{p^i},N]
  \pmod{\gamma_{p+1}(G)^p\gamma_{p+2}(G)}.
  $$
  This yields
  $$
  [N^{p^i}\gamma_p(N),G] \leq [G^{p^i}\gamma_p(G),N] \cdot
  [\gamma_p(G),G]^p\cdot [\gamma_p(G),G,G].
  $$
\end{proof}

\begin{eje}
  Let $p$ be a prime, and let $H = \langle \alpha \rangle \ltimes A$
  be the pro-$p$ group of maximal class, where $\langle \alpha \rangle
  \cong C_p$, $A = \langle x_1, \ldots, x_{p-1} \rangle \cong
  \Z_p^{p-1}$, and the action of $\alpha$ on $A$ is given by
  $$
  [x_i,\alpha] = x_{i+1} \text{ for $1 \leq i \leq p-2$,} \quad
  \text{and} \quad [x_{p-1},\alpha] = \prod_{j=1}^{p-1} x_j^{-\binom{p}{j}}.
  $$
  Let $G := H/[H^p,H]^p[H^p,H,H]$.  Then $G$ is a finite $p$-group of
  order $p^{p+2}$ and the word $x^p$ is not interchangeable in $G$.

  Indeed, consider the image $N$ of $A$ in $G$.  Since $\lvert G/N^p
  \rvert = p^p$, we have $\gamma_p(G/N^p) = 1$.  Hence $G/N^p$ is
  regular and $G/N^p$ has exponent $p$.  This shows that $N^p = G^p$.
  Now one easily verifies that $[N^p,G] = [G^p,G]$ is a non-trivial
  central cyclic subgroup of $G$.  This shows that
  $$
  [N^p,G] \not \subseteq 1 = [N,G^p] [G^p,G]^p [G^p,G,G].
  $$
\end{eje}

We are ready to prove the main result of this section.

\begin{teo}
  \label{thm:thompson}
  Let $w$ be a word, and let $G$ be a $w$-maximal finite $p$-group
  such that $w$ is interchangeable in $G$.  Then one has $w(G) \leq
  \Zen (G)$.
\end{teo}

\begin{proof}
  For a contradiction, assume that $G$ is a minimal counterexample.
  Then $[w(G),G]$ is cyclic of order $p$ and contained in
  the centre $\Zen(G)$ of $G$, i.e.\ $[w(G),G]^p [w(G),G,G] = 1$.
  Consider the following characteristic subgroups of $G$:
  \begin{align*}
    N_1 & := \{ x \in G \mid [x,w(G)] = 1\} = \cen_G(w(G)), \\
    N_2 & := \{ x \in w(G) \mid [x,G] = 1\} = \Zen(G) \cap w(G).
  \end{align*}

  We observe that $N_2 \leq N_1$.  Since $w$ is interchangeable in
  $G$, we conclude that $[w(N_1),G] \leq [N_1,w(G)] = 1$.  In
  particular $w(N_1) \leq N_2$.

  Next we define $\langle , \rangle :G/N_1 \times w
  (G)/N_2\longrightarrow [w (G),G]\cong \F_p$ given by $\langle
  xN_1,yN_2\rangle =[x,y]$.  This map $\langle ,\rangle$ is a pairing
  of abelian $p$-groups, and hence $\lvert G:N_1 \rvert = \lvert
  w(G):N_2 \rvert$.  Therefore $\lvert N_1 : w(N_1) \rvert \geq \lvert
  N_1:N_2 \rvert = \lvert G:w(G) \rvert$, and $G$ is not $w$-maximal,
  the desired contradiction.
\end{proof}

Of course, Question~\ref{qu:max-w-max} specialises to

\begin{que}
  Let $F$ be the free group on generators $x_1, \ldots, x_n$, and let
  $w$ be a commutator word in $F$, i.e.\ let $w \in [F,F]$.  Can one
  classify $w$-maximal $p$-groups of a given $w$-breadth $m$ which are
  maximal with respect to the partial ordering $\preceq_w$?
\end{que}

Other natural questions arising from our discussion are

\begin{que}
  Characterise words which are interchangeable in all $w$-maximal
  $p$-groups.
\end{que}

\begin{que}
  Let $w$ be a word and let $G$ be a $w$-maximal group.  If $w$ is
  interchangeable in all $w$-maximal $p$-groups, we know from
  Theorem~\ref{thm:thompson} that $w(G)$ is contained in the center $\Zen(G)$
  of $G$.  Can one describe properties of the inclusion of $w(G)
  \subseteq G$ in other situations?
\end{que}

Next we give some applications of $w$-maximal $p$-groups which
allow us to translate, in any finite $p$-group $G$, information
about the size of quotients with certain properties to information
about the maximal size of subgroups with similar properties.
Indeed, Proposition~\ref{pro:large_subgroups} guarantees the
existence of large subgroups of comparatively small nilpotency
class, i.e.\ of order at least $\lvert G : \gamma_c(G) \rvert$ and
class at most $c$. In a similar vein,
Proposition~\ref{pro:epsilon} shows that there exist subgroups of
nilpotency class $2$ and exponent $p^i$ which are of size at least
$\lvert G : G^{p^i}[G,G] \rvert$.  Corollary~\ref{cor-2} is an
extension of the result of Laffey, which we quoted in the
introduction; see~\cite{La73}.

\begin{pro}\label{pro:large_subgroups}
  Let $G$ be a finite $p$-group and $c \in \N$.  Then there exists a
  subgroup $H$ of $G$ of nilpotency class at most $c$ such that
  $\lvert H \rvert \geq \lvert G : \gamma_c(G) \rvert$.
\end{pro}

\begin{proof}
  Clearly, there exists a $\gamma_c$-maximal subgroup $H$ of $G$ such
  that $\lvert G : \gamma_c(G) \rvert \leq \lvert H : \gamma_c(H)
  \rvert$.  From Lemma~\ref{swap} and Theorem~\ref{thm:thompson} we deduce
  that $\gamma_{c+1}(H) = 1$.
\end{proof}

\begin{pro} \label{pro:epsilon} Let $G$ be a finite $p$-group and let $i \in
  \N$ with $i \geq 2$ if $p=2$.  Put $\epsilon := 0$ if $p$ is odd,
  and $\epsilon := 1$ if $p=2$.  Then there exists a subgroup $H$ of
  $G$ of nilpotency class at most $2$ and exponent $p^{i+\epsilon}$
  such that $\lvert H \rvert \geq \lvert G:G^{p^{i + \epsilon}}[G,G]
  \rvert$.
\end{pro}

\begin{proof}
  Put $w =x^{p^i}[y,z]$.  By induction on the order of $G$ we may
  assume that $G$ is $w$-maximal.  Recall the notation $\Omega_i(G) :=
  \langle x \in G \mid x^{p^i}=1 \rangle$.  Consider first the case
  when $p$ is odd.  Lemma~\ref{swap} and Theorem~\ref{thm:thompson} show that
  $[G,G,G]=1$.  Therefore $G$ is a regular $p$-group, and
  $$
  \lvert G:G^{p^i}[G,G] \rvert \leq \lvert G:G^{p^i} \rvert = \lvert
  \Omega_i(G) \rvert,
  $$
  where $H := \Omega_i(G) = \{ x \in G \mid x^{p^i}=1 \}$ is a
  subgroup of $G$ of exponent $p^i$ and nilpotency class at most $2$;
  see \cite[Kapitel III \S 10]{Hu67}.

  Now suppose that $p=2$. Since $G$ is $w$-maximal, Lemma~\ref{swap}
  and Theorem~\ref{thm:thompson} yield $[G,G,G]=1$ and $[G,G]^{2^i} =
  [G^{2^i},G] = 1$.  By the Hall-Petrescu identity we have
  $(xy)^{2^{i+1}} = x^{2^{i+1}} y^{2^{i+1}}$ for all $x, y \in G$; see
  \cite[Kapitel III, Satz 9.4]{Hu67}.  Therefore $\phi : G \rightarrow
  G^{2^{i+1}}$, $x \mapsto x^{2^{i+1}}$ is a surjective homomorphism
  with kernel $\Omega_{i+1}(G) = \{x \in G \mid x^{2^{i+1}} = 1\}$.
  Hence $\lvert \Omega_{i+1}(G) \rvert = \lvert G : G^{2^{i+1}} \rvert
  \geq \lvert G : G^{2^{i+1}} [G,G] \rvert$, and $\Omega_{i+1}(G)$ is
  a subgroup of $G$ of exponent $2^{i+1}$ and nilpotency class at most
  $2$.
\end{proof}

A result of Glauberman \cite{Gl08} allows us to deduce the existence
of normal subgroups with similar properties.

\begin{cor}
  Let $G$ be a finite $p$-group and let $i \in \N$.  Suppose that $p
  \geq 7$.  Then there exists a normal subgroup $H$ of $G$ of
  nilpotency class at most $2$ and exponent $p^i$ such that $\lvert H
  \rvert = \min \{ \lvert G:G^{p^{i + \epsilon}}[G,G] \rvert,
  p^{\lfloor (2p+4)/3 \rfloor} \}$.
\end{cor}

\begin{cor} \label{cor-2} Let $G$ be a finite $p$-group. Let $p^k$ be
  the maximal order of a subgroup of $G$ of nilpotency class $2$ and
  exponent $p$, if $p$ is odd, nilpotency class $2$ and exponent $8$,
  if $p=2$.  Then $G$ can be generated by $k$ elements, i.e.\ $d(G)
  \leq k$.
\end{cor}

\begin{proof}
  The claim follows from $d(G) = \log_p \lvert G:G^p[G,G] \rvert$ and
  Proposition~\ref{pro:epsilon}.
\end{proof}

\section{$d$-Maximal finite $p$-groups and $\Z_p$-Lie
  rings} \label{sec:Lie}

For any group $G$, let $d(G)$ denote the minimal number of
elements required to generate $G$, possibly $\infty$.  Throughout
this section let $p$ be a prime.  As noted in the introduction a
finite $p$-group $G$ is $d$-maximal, i.e.\ satisfies $d(H) < d(G)$
for all proper subgroups $H$ of $G$, if and only if it is
$w$-maximal for $w = x^p[y,z]$.  The following proposition (cf.\
\cite{Ka91}) is an easy consequence of Theorem~\ref{thm:thompson}.

\begin{pro} \label{pro:d-maximal} Let $p$ be an odd prime and let $G$
  be a $d$-maximal finite $p$-group.  Then $G/[G,G]$ and $[G,G]$ are
  elementary abelian $p$-groups.  If $\lvert G \rvert > p$, then $\lvert
  [G,G] \vert \leq p^{d(G) - 2}$.
\end{pro}

\begin{proof}
  Since $G$ is $d$-maximal, $G/[G,G]$ is also $d$-maximal and
  therefore an elementary abelian $p$-group.  By Lemma~\ref{swap} and
  Theorem~\ref{thm:thompson}, one has $[G,G,G]=1$ and $[G,G]^p = [G^p,G] =
  1$.  Hence $[G,G]$ is an elementary abelian $p$-group.

  Suppose that $\lvert G \rvert > p$, and put $d := d(G)$.  Since $G$
  is $d$-maximal, we have $\lvert [G,G] \vert \leq p^{d-1}$.  For a
  contradiction, assume that $\lvert [G,G] \rvert = p^{d-1}$.  The map
  $\pi :G/[G,G] \to [G,G]$, $x \mapsto x^p$ is a homomorphism between
  elementary abelian $p$-groups.  Since $\lvert G/[G,G] \rvert >
  \lvert [G,G] \rvert$, the map $\pi$ is not injective.  Suppose that
  $1 \neq y \in \ker \pi$.  Then $y$ is an element of order $p$ in $G$
  and $y \not \in [G,G]$.  Put $H = \langle \{ y \} \cup [G,G]
  \rangle$. Then $H$ is an elementary abelian $p$-group of order $p^d$
  which is strictly contained in $G$, a contradiction.
\end{proof}

For $p$ odd, Theorem~\ref{thm:thompson} and
Proposition~\ref{pro:d-maximal} show that every $d$-maximal finite
$p$-group $G$ has nilpotency class at most $2$ and exponent $p^2$.
Conversely, finite $p$-groups of nilpotency class at most $2$ and
exponent $p$ need not be $d$-maximal; for instance, the Heisenberg
group over the finite field $\F_p$ is not $d$-maximal.
Interesting examples of $d$-maximal finite $p$-groups can be
constructed by Lie methods as follows.

The Lazard correspondence, which operates via functors $\exp$ and
$\log$, is a correspondence between finite $p$-groups of nilpotency
class smaller than $p$ and finite $\Z_p$-Lie rings of nilpotency class
smaller than $p$; see \cite[\S~10.2]{Kh73}.  Suppose that $p$ is odd,
and consider a finite $p$-group $G$ whose nilpotency class is bounded
by $2$.  Put $L := \log (G)$.  Then the group $G$ is $d$-maximal if
and only if $L$ is $d$-maximal, in the sense that for any Lie subring
$M$ of $L$ one has $\lvert M : pM + [M,M]_{\text{Lie}} \rvert < \lvert
L : pL + [L:L]_{\text{Lie}} \rvert$.

\begin{eje} \label{ej:example1} Let $L = \Span \langle x, y, z \mid px
  = py = p^2 z = 0 \rangle \cong C_p\times C_p \times C_{p^2}$ be and
  abelian $p$-group, and extend $[x,z]_{\text{Lie}} =
  [y,z]_\text{{Lie}} = 0$ and $[x,y]_{\text{Lie}} = pz$ bi-additively
  to $L$.  Then $(L,+,[,]_{\text{Lie}})$ is a finite $d$-maximal
  $\Z_p$-Lie ring.  The finite $p$-group $G = \exp (L)$ is $d$-maximal
  such that $d(G) = 3$ and $\lvert [G,G] \rvert = p^2$.
\end{eje}

We continue to work under the hypothesis that $p>2$ and consider a
finite $p$-group $G$ of nilpotency class $2$ and of exponent $p$.
Then $L := \log(G)$ is an $\F_p$-Lie algebra and, as an
$\F_p$-vector space, $L$ decomposes as $L = V \oplus
[L,L]_{\text{Lie}}$.  Put $k := \dim_{\F_p} [L,L]_{\text{Lie}}$.
Since $[L,L]_{\text{Lie}} \subseteq \Zen(L)$, the Lie product on
$L$ is determined by its restriction to $V \times V$, which can be
written as
\begin{equation}
  V \times V \to [L,L]_{\text{Lie}}, \quad
  [v,w]_{\text{Lie}} = f_1(v,w)z_1 + \ldots + f_k(v,w)z_k,
\end{equation}
where $\{ z_1,\ldots ,z_k\}$ is an $\F_p$-basis of
$[L,L]_{\text{Lie}}$ and $f_i$, $1 \leq i \leq k$ , is a collection of
antisymmetric bilinear forms on the vector space $V$.

In order to check whether the $p$-group $G$ is $d$-maximal, it is
enough to check whether the $\F_p$-Lie algebra $L$ is $d$-maximal.
Clearly, for this is suffices to check whether for any Lie subalgebra
$M$ of $L$ containing $[L,L]_{\text{Lie}}$ one has $\lvert M :
[M,M]_{\text{Lie}} \rvert < \lvert L : [L,L]_{\text{Lie}} \rvert$.
Equivalently, one needs to test whether for any proper subspace $W$ of
$V$,
\begin{equation}
  \dim (W) + k - \dim_{\F_p} \left( \Span \langle \sum_{i=1}^k f_i(v,w) z_i
    \mid v, w\in W \rangle \right) < \dim_{\F_p} V.
\end{equation}
One can easily compute the dimension of $\Span \langle
\sum_{i=1}^kf_i(v,w) z_i \mid v, w \in W \rangle$ by studying the
space generated by the antisymmetric bilinear forms $f_1,\ldots ,f_k$
in the exterior algebra $W \wedge W$.

\begin{lem}\label{lem:wedge}
  Let $W$ be an $\F_p$-vector space, let $f_1, \ldots, f_k$ be a
  collection of antisymmetric bilinear forms on $W$, and let $\{z_1,
  \ldots, z_k \}$ be a basis of $\F_p^k$.  Let $U$ denote the vector
  subspace generated by $f_1, \ldots, f_k$ in the exterior algebra $W
  \wedge W$. Then
  \begin{equation}
    \dim_{\F_p} U = \dim_{\F_p} \left( \Span \langle \sum_{i=1}^k
      f_i(v,w) z_i \mid v,w \in W \rangle \right).
  \end{equation}
\end{lem}

\begin{proof}
  For $\boldsymbol{\lambda} := (\lambda_1, \ldots, \lambda_k) \in
  \F_p^k$ define $\Phi_{\boldsymbol{\lambda}} : \F_p^k \to \F_p$,
  $\sum_{i=1}^k x_i z_i \mapsto \sum_{i=1}^k \lambda_i x_i$.  Then
  $\sum_{i=1}^k \lambda_i f_i = 0$ is a linear dependency relation in
  $W \wedge W$ if and only if for all $v,w \in W$,
  $\Phi_{\boldsymbol{\lambda}}(\sum_{i=1}^k f_i(v,w) z_i) = 0$.
\end{proof}

\begin{eje} \label{ej:example2} Suppose that $p$ is odd.  Let $V =
  \Span \langle e_1, \ldots, e_4 \rangle$ be the standard
  $4$-dimensional $\F_p$-vector space and consider the antisymmetric
  bilinear forms $f_1$ and $f_2$ on $V$, represented by the matrices
  \begin{equation}
    F_1 =
    \begin{pmatrix}
      0 & 1& 0 &0 \\
      -1 & 0& 0 &0 \\
      0 & 0& 0 &1 \\
      0 & 0& -1 &0
    \end{pmatrix}
    \quad \text{and} \quad
    F_2 =
    \begin{pmatrix}
      0 & 0& 0 &1 \\
      0 & 0& 1 &0 \\
      0 & -1& 0 &1 \\
      -1 & 0& -1 &0
    \end{pmatrix},
  \end{equation}
  with respect to the basis $(e_1, \ldots, e_4)$.  Then the
  $6$-dimensional $\F_p$-Lie algebra $L = V \oplus \Span \langle
  z_1,z_2 \rangle$, defined by $[v,w]_{\text{Lie}} = f_1(v,w) z_1 +
  f_2(v,w)z_2$, is of nilpotency class $2$.  Based on
  Lemma~\ref{lem:wedge}, a short computation shows that $L$ is
  $d$-maximal.  The finite $p$-group $G = \exp (L)$ is $d$-maximal
  such that $d(G) = 4$ and $\lvert [G,G] \rvert = p^2$.
\end{eje}

Examples \ref{ej:example1} and \ref{ej:example2} suggest that,
perhaps, the bound in Proposition \ref{pro:d-maximal} is the best
possible.  We record

\begin{que}
  Does there exist for every integer $k>2$ a $d$-maximal finite
  $p$-group $G$ such that $\lvert G : \Phi (G) \rvert = p^k$ and
  $\lvert [G,G] \rvert = p^{k-2}$?
\end{que}

\section{Hereditarily $w$-maximal groups} \label{sec:hereditarily_w}

Let $w = w(\mathbf{x})$ be a fixed word, i.e.\ an element of the free
group $F = \langle x_1,\ldots,x_n \rangle$ on $n$ generators.  We say
that a group $G$ is \emph{hereditarily $w$-maximal} if all subgroups
of $G$ are $w$-maximal.

\begin{pro}\label{pro:vanishes}
  Let $w$ be an element of the free group $F = \langle x_1,\ldots,x_n
  \rangle$ and let $G$ be a hereditarily $w$-maximal finite group.
  \begin{enumerate}
  \item[\textup{(1)}] If $G$ is a $p$-group, then $w (G) = 1$.
  \item[\textup{(2)}] If $w$ is equal to $x^m$ or $\gamma_k = [x_1,
    \ldots, x_k]$, for some $m,k \in \N$, then $w (G)=1$
  \end{enumerate}
\end{pro}

\begin{proof}
  (1) We argue by induction on the order of $G$.  Suppose that $G$ is
  a non-trivial $p$-group.  Let $H$ be a maximal subgroup of $G$,
  i.e.\ a subgroup of index $p$ in $G$.  Since $H$ is hereditarily
  $w$-maximal, induction shows that $w(H) = 1$.  From $\lvert G : w(G)
  \rvert > \lvert H : w(H) \rvert = \lvert H \rvert$, we conclude that
  $w(G) = 1$.

  (2) If $w = x^m$ for some $m \in \N$, then for every $g \in G$, the
  cyclic subgroup $\langle g \rangle$ is $x^m$-maximal, and hence $g^m
  = 1$.

  Next suppose that $w = \gamma_k = [x_1, \ldots, x_k]$ for some $k
  \in \N$. If $G$ is nilpotent, then it is a direct product of its
  Sylow $p$-subgroups and the claim follows from (1).  For a
  contradiction, assume that $G$ is not nilpotent.  By induction, we
  may assume that all proper subgroups of $G$ are nilpotent of class
  at most $k-1$.  By a classical result of Iwasawa (see \cite{Iw41}),
  $G \cong C \ltimes Q$ where $C$ is a cyclic $p$-group and $Q$ is a
  $q$-group with $p$ and $q$ prime.  There are two cases.

  \noindent \textit{Case 1: $[G,G] \lneqq Q$.} Then there exist
  subgroups $H_1$ and $H_2$ of index $p$ and $q$, and these are
  nilpotent of class at most $k-1$.  Since $\lvert G : \gamma_k(G)
  \rvert > \lvert H_i : \gamma_k(H_i) \rvert = \lvert H_i \rvert$ for
  $i \in \{1,2\}$, we conclude that $\gamma_k(G) = 1$.

  \noindent \textit{Case 2: $[G,G] = Q$.}  In this case, since $C$ is
  cyclic, $[G,G]=[G,Q]=Q$. Therefore $\gamma_k(G) = Q$ and $\lvert
  G:\gamma_k(G) \rvert = \lvert G:Q \rvert = \lvert C \rvert = \lvert
  C : \gamma_k(C) \rvert$, a contradiction.
\end{proof}

The following example shows that there is no direct analogue of
Proposition~\ref{pro:vanishes} for the second commutator word $w =
[[x_1,x_2],[y_1,y_2]]$.

\begin{eje}\label{eje:quaternion}
  Consider the quaternion group $Q_8 = \{\pm 1, \pm i ,\pm j ,\pm k\}$
  of order $8$ and a cyclic group $C_3 = \langle \alpha \rangle$ of
  order $3$.  Take the semidirect product $G = C_3 \ltimes Q_8$ with
  respect to the natural action, given by $i^\alpha =j$, $j^\alpha =k$
  and $k^\alpha =i$.  We have $G^{(2)} = [[G,G],[G,G]] = \{\pm 1\}$.
  This shows that $\lvert G:G^{(2)} \rvert = 12$, and since $G$ has no
  proper subgroup of order larger than or equal to $12$, the group $G$
  is $[[x_1,x_2],[y_1,y_2]]$-maximal.  But $G$ is not metabelian.
\end{eje}

We show that the special example~\ref{eje:quaternion} generalises to
higher commutator words.  The \emph{standard derived words} are
defined recursively as
$$
\delta_1 := [x,y] \quad \text{and} \quad \delta_{k+1} :=
[\delta_k(x_1, \ldots, x_{2^k}), \delta_k(y_1, \ldots, y_{2^k})]
\text{ for $k \in \N$.}
$$
Accordingly, a group $G$ is soluble of derived length at most $k$ if
and only if $\delta_k(G) = 1$.

\begin{teo}\label{thm:example}
  Let $k \in \N$.  Then there exists a finite group $G$ which is
  $\delta_{k+1}$-maximal, soluble of derived length $k+2$, but
  satisfies $\delta_{k+1}(G) \not = 1$.
\end{teo}

\begin{proof}
  Let $p,q$ be prime numbers with $2^k < p < q$, and let $m$ be the
  order of $p$ in $(\Z/q\Z)^*$.  Then $\F_p(\zeta) \cong \F_{p^m}$,
  where $\zeta$ is a $q$th root of unity.

  Let $L$ be the free nilpotent Lie algebra of nilpotency class $2^k$
  over $\F_p$ on $d := 2m$ generators $x_1, \ldots, x_d$.  Recall that
  the derived series of $L$ is a subseries of the lower central
  series: $\delta_j(L) = \gamma_{2^j}(L)$ for all $j \in \N_0$.  Hence
  $L$ has derived length $k+1$, with $\delta_k(L) = \gamma_{2^k}(L)
  \not = 0$ central in $L$.

  Let $V := \F_p x_1 + \ldots + \F_p x_d$ so that $L = V \oplus
  \gamma_2(L)$ as an $\F_p$-vector space.  Write $V = V_1 \oplus
  V_{-1}$ with $V_i \cong \F_p(\zeta)$ for $i \in \{1,-1\}$ and choose
  $z_V \in \textup{GL}(V)$ of prime order $q$, acting on $V_i$ as
  multiplication by $\zeta^i$ for $i \in \{1,-1\}$.  Since $L$ is
  free, any vector space automorphism of $V$ lifts uniquely to a Lie
  algebra automorphism of $L$.  We denote the lift of $z_V$ to
  $\textup{Aut}(L)$ by $z_L$.  Clearly, $\zeta$ and $\zeta^{-1}$ are
  among the eigenvalues of the automorphism $z_V$ of $V$.  For later
  use we observe that the $\F_p \langle z_V \rangle$-module $V$ is
  completely reducible with irreducible submodules $V_1$ and $V_{-1}$.
  In particular, $V$ does not admit any $z_V$-invariant subspaces of
  co-dimension $1$.

  Clearly, $z_L$ acts on the $\F_p$-vector space $\delta_k(L)$ and we
  denote the restriction of $z_L$ to $\delta_k(L) = \gamma_{2^k}(L)$
  by $z_{\delta_k(L)}$.  The eigenvalues of $z_{\delta_k(L)}$ are
  products of length $2^k$ in the eigenvalues of $z_V$; among the
  latter are $\zeta$ and $\zeta^{-1}$.  This shows that $1$ is an
  eigenvalue of $z_{\delta_k(L)}$.  Since the $\F_p \langle
  z_{\delta_k(L)} \rangle$-module $\delta_k(L)$ is completely
  reducible, we find a subalgebra $Z$ of co-dimension $1$ in
  $\delta_k(L)$, which is $z_L$-invariant.  Since $\delta_k(L)$ lies
  in the centre of $L$, the subalgebra $Z$ is, in fact, an ideal of $L$.

  As $p > 2^k$, Lazard's correspondence yields a finite $p$-group $N
  := \exp(L/Z)$ (of exponent $p$) with a natural action of $\langle z
  \rangle \cong C_q$ on $N$, where $z = z_L$.  Under this
  correspondence, $V \cong (V+Z)/Z$ is isomorphically mapped to a
  complement of $\gamma_2(N)$ in $N$; we denote this complement also
  by $V$ so that $N = V \ltimes \gamma_2(N)$.  We claim that the
  semidirect product $G := \langle z \rangle \ltimes N$ is
  $\delta_{k+1}$-maximal, while $\delta_{k+1}(G) \not = 1$.

  Indeed, since $z_V$ does not admit $1$ as an eigenvalue, we have $N
  \supseteq \delta_1(G) \supseteq [V,\alpha] = V$.  As $N = \langle V
  \rangle$, it follows that $\delta_1(G) = N$, and hence
  $\delta_{k+1}(G) = \delta_k(N)$, which is associated under the Lazard
  correspondence to $\delta_k(L/Z)$, is a cyclic group of order $p$
  and hence non-trivial.  This implies that $\lvert G : \delta_{k+1}(G)
  \rvert = \lvert G:N \rvert \lvert N:\delta_k(N) \lvert = \lvert G \rvert /
  p$.

  Now suppose that $H$ is a subgroup of $G$ with $\lvert H :
  \delta_{k+1}(H) \rvert \geq \lvert G : \delta_{k+1}(G) \rvert$.  A
  fortiori we have $\lvert H \rvert \geq \lvert G \rvert / p$.  Since
  $q > p$ is prime, this implies that $H \not \subseteq N$ and $HN =
  G$.  Then $\lvert N : H \cap N \rvert \leq p$, and consequently $(H
  \cap N)\gamma_2(N)$ has index at most $p$ in $N/\gamma_2(N) \cong
  V$.  Since $V$ does not admit any $z$-invariant subspaces of
  co-dimension $1$, we must have $(H \cap N)\gamma_2(N) = N$.  This
  implies that $H \cap N = N$, and $H = G$.  Thus $G$ is
  $\delta_{k+1}$-maximal.
\end{proof}

We finish this section by proving that, for many words $w$, finite
groups which are hereditarily $w$-maximal are necessarily soluble.  In
particular, the proposition shows that a hereditarily
$\delta_k$-maximal groups are solvable.

\begin{pro}
  Let $w$ be an element of the free group $F = \langle x_1,\ldots,x_n
  \rangle$ and let $G$ be a hereditarily $w$-maximal finite group.  Then
  $w$ vanishes on every composition factor of~$G$.
\end{pro}

\begin{proof}
  Consider a descending chain $G = N_0 \trianglerighteq N_1
  \trianglerighteq \ldots \trianglerighteq N_k = 1$ of subnormal
  subgroups of $G$ such that $N_{i-1}/N_i$ is simple for each $i \in
  \{1, \ldots, k \}$. By induction on the composition length it is
  enough to prove that $w$ vanishes on $G/N_1$.  For a contradiction,
  assume that $w$ does not vanish on $G/N_1$ so that, in particular $G
  = w(G) N_1$.  Then $\lvert G:w (G) \rvert = \lvert N_1:N_1\cap w (G)
  \rvert \leq \lvert N_1:w (N_1) \rvert$, which contradicts the
  $w$-maximality of~$G$.
\end{proof}

\begin{que}
  Is there a uniform bound on the derived length of a hereditarily
  $\delta_2$-maximal group, where $\delta_2 = [[x_1,x_2],[y_1,y_2]]$?
\end{que}

\section{Hereditarily $d$-maximal groups} \label{sec:hereditarily_d}

More definitive results can be obtained for hereditarily $d$-maximal
groups which we define as follows.  A group $G$ is said to be
\emph{hereditarily $d$-maximal}, if every subgroup $H$ of $G$ is
$d$-maximal.  Of course, a finite $p$-group $G$ is hereditarily
$d$-maximal if and only if $G$ is hereditarily $w$-maximal for $w =
x^p[y,z]$.

\begin{lem}\label{lem2}
  Let $G$ be a $d$-maximal group. Then $G$ is finitely generated, and
  for every maximal subgroup $M$ of $G$ we have $d(G) = d(M) + 1$.
\end{lem}

\begin{proof}
  If $G$ is trivial there is nothing to show. Now suppose that $M$ is
  a maximal subgroup of $G$. Since $G$ is $d$-maximal, the inequality
  $d(G) \geq d(M) + 1$ holds and in particular $d(M)$ is finite. On
  the other hand, if $g \in G \setminus M$, then the maximality of $M$
  implies $\langle M \cup \{g\} \rangle = G$, thus $d(G) \leq d(M) +
  1$.
\end{proof}

\begin{lem}\label{lem4}
  Let $G$ be a hereditarily $d$-maximal group. Then $G$ is
  equi\-chained of finite length, i.e.\ there exists a finite chain $1
  = M_0 \lneqq M_1 \lneqq \ldots \lneqq M_r = G$, where $M_i$ is
  maximal in $M_{i+1}$ for all indices $i$, and all such chains have
  the same length. Moreover this length is equal to $d(G)$.
\end{lem}

\begin{proof}
  This follows by induction from Lemma~\ref{lem2}.
\end{proof}

The existence of so-called Tarski groups, i.e.\ infinite groups all of
whose non-trivial proper subgroups have prime order $p$, indicates
that it may be difficult to classify hereditarily $d$-maximal groups
in general; see \cite{Ol80,Ol82} for a construction of such groups.

In order to avoid these problems we choose to restrict our attention
to residually-finite groups. Note that every non-trivial
residually-finite group has at least one maximal subgroup of finite
index. Thus, if $G$ is a residually-finite $d$-maximal group, then
Lemma~\ref{lem2} and an easy induction on $d(G)$ show that $G$ is in
fact finite. It remains to present a classification of \emph{finite}
hereditarily $d$-maximal groups.

For $n \in \N$ let $\nu(n)$ denote the number of prime divisors,
counting repetitions, of $n$.  By a classical result of Iwasawa, a
finite group is equi\-chained if and only if it is super\-soluble; see
\cite[Satz VI.9.7]{Hu67} or \cite{Iw41b}. This gives

\begin{pro}\label{prop5}
  Let $G$ be a finite group. Then $G$ is hereditarily $d$-maximal if
  and only if $d(G) = \nu( \lvert G \rvert )$.
\end{pro}

\begin{proof}
  Suppose that $G$ is $d$-maximal. Then Iwasawa's characterisation of
  finite equi\-chained groups and Lemma~\ref{lem4} imply that $G$ is
  super\-soluble and $d(G) = \nu(\lvert G \rvert )$.

  Now suppose that $d(G) = \nu( \lvert G \rvert )$. Then for every
  subgroup $H \leq G$ we have $d(H) = \nu( \lvert H \rvert )$. It
  follows that $G$ is hereditarily $d$-maximal.
\end{proof}

\begin{cor}\label{cor6}
  Every quotient of a finite hereditarily $d$-maximal group is
  hereditarily $d$-maximal.
\end{cor}

\begin{lem}\label{lem7}
  Let $G$ be a finite nilpotent group. Then $G$ is hereditarily
  $d$-maximal if and only if $G$ is elementary abelian.
\end{lem}

\begin{proof}
  Suppose that $G$ is hereditarily $d$-maximal. Being nilpotent, the
  group $G$ is the direct product of its Sylow subgroups, $G = P_1
  \times \ldots \times P_r$ say. Then $d(G) = \max \{ d(P_i) \mid 1
  \leq i \leq r \}$, and since $G$ is $d$-maximal, this implies that
  $G$ is a $p$-group for a suitable prime $p$. Proposition~\ref{prop5}
  asserts that $\log_p \lvert G/\Phi(G)\rvert = d(G) = \nu( \lvert G
  \rvert ) = \log_p \lvert G \rvert$, so $G$ is equal to its Frattini
  quotient, thus an elementary $p$-group.

  It is clear that elementary abelian groups are hereditarily
  $d$-maximal.
\end{proof}

\begin{teo} \label{thm:classif}
 Let $G$ be a finite group. Then $G$ is hereditarily $d$-maximal if
 and only if one of the following holds:
 \begin{enumerate}
 \item[\textup{(1)}] there exist primes $p,q$ such that $G = \langle x
   \rangle \ltimes Q$ where $\langle x \rangle \cong C_p$, $Q$ is an
   elementary $q$-group and $x$ acts on $Q$ as a non-trivial scalar;
 \item[\textup{(2)}] there exists a prime $p$ such that $G$ is an elementary
   $p$-group.
 \end{enumerate}
\end{teo}

\begin{proof}
  It is clear that the groups given in the list are hereditarily
  $d$-maximal. For the opposite direction, suppose that $G$ is
  hereditarily $d$-maximal. In the proof of Proposition~\ref{prop5} it
  was seen that $G$ is super\-soluble. This implies that the derived
  subgroup $D := [G,G]$ of $G$ is nilpotent. Corollary~\ref{cor6} and
  Lemma~\ref{lem7} show that $G/D$ and $D$ are elementary abelian,
  i.e.\ $G/D \cong C_p^r$ and $D \cong C_q^s$ for primes $p,q$ and
  $r,s \in \N_0$.

  If $D = 1$, there is nothing more to prove. So assume that $r,s \geq
  1$. Choose generators $x_1, \ldots, x_r$ for $G$ modulo $D$.
  Proposition~\ref{prop5} gives $d(G) = \nu(\lvert G \rvert ) = r +
  s$, which has strong consequences:
  \begin{enumerate}
  \item[(i)] every non-trivial element of $G$ has order $p$ or $q$;
  \item[(ii)] the group $H := \langle x_1, \ldots, x_r \rangle$ has
    order $p^r$, in particular $G = H \ltimes D$ is a split extension
    of $D$ by $H$;
  \item[(iii)] for every $y \in D \setminus \{1\}$ the group $\langle
    y \rangle \cong C_q$ is normal in $G$ and its centraliser
    $\textup{C}_H(\langle y \rangle)$ in $H$ is trivial; consequently
    $H = \langle x_1 \rangle \cong C_p$ and $x_1$ acts as a
    non-trivial scalar on $D$.
  \end{enumerate}
\end{proof}

%%%%%

\end{document}